\documentclass[preprint,12pt]{elsarticle}
 
\usepackage{amsmath,amssymb,amsfonts,amsthm}
\usepackage{graphicx}
\usepackage{textcomp}
\usepackage{xcolor}
\usepackage{xspace}
\usepackage{booktabs}
\usepackage{algorithm}
\usepackage{algpseudocode}
\usepackage{xurl}
\usepackage{appendix}
\usepackage{tabularx} 
\usepackage{booktabs}
\allowdisplaybreaks[4]
 
\newtheorem{theorem}{Theorem}

\newtheorem{corollary}{Corollary}

\definecolor{lightgray}{gray}{0.9}
\definecolor{darkergray}{gray}{0.8}
\definecolor{forestgreen}{RGB}{34,139,34}

\usepackage{hyperref}
\hypersetup{
    colorlinks=true,
    citecolor=forestgreen,
    filecolor=black,
    linkcolor=red,
    urlcolor=blue
}
\usepackage[capitalize]{cleveref}
 
\DeclareMathOperator{\OPT}{OPT}
\DeclareMathOperator{\ALG}{ALG}

\journal{International Journal of Electrical Power \& Energy Systems}
 
\begin{document}
 
\begin{frontmatter}
 
\title{Fast Relax-and-Round Unit Commitment With Economic Horizons\tnoteref{t1}}
 
\author[ornl]{Shaked Regev\corref{cor1}}
\author[ornl]{Evan J. R. Brody}
\author[ornl]{Charles Foltz}
\author[ornl]{Eve Tsybina}
\author[ornl]{Slaven Pele\v{s}}
  
\affiliation[ornl]{organization={Oak Ridge National Laboratory},
            addressline={1 Bethel Valley Road},
            city={Oak Ridge},
            state={Tennessee},
            country={USA}}

\begin{abstract}
The US energy system is under increasing pressure to serve expanding data centers loads and to accommodate a larger number of generating units with varying technologies and ownership structures. Therefore, developing new unit commitment  (\textbf{UC}) methods remains a priority for reliable and affordable grid operations.
We expand our novel computational method for UC to include ramping constraints and long-horizon planning, and provide a theoretical bound on its accuracy. We solve problems with thousands of generators at 5-minute market intervals. We show that our method can solve UC problems with over 20,000 generators in approximately 10 seconds on commodity hardware and that an increased planning horizon leads to sizable operational cost savings. This is twice as fast as the leading commercial solver, Gurobi. The gap increases when adding more constraints. Our method can be implemented using existing continuous optimization solvers and adapted for different applications. 
We prove a bound on the error of these solvers and show that it vanishes (in relative terms) as the problem becomes larger. Our novel algorithm could allow an operator to make horizon-aware economic decisions for large systems.\end{abstract}
 
 
\end{frontmatter}
 
\section{Introduction}
\label{sec:introduction}

Electric power systems are under increasing operational pressure due to large load volatility caused by increasing demand for electricity and the rapid expansion of data centers. The North American Electric Reliability Corporation (\textbf{NERC}) classifies data centers as a distinct load category, recognizing their unique planning, operational, and balancing risks \cite{NERC_loads}. Unlike traditional industrial loads, data centers can ramp up or down rapidly in response to computational workloads, creating steep intra-hour load changes \cite{Frontier2024}. 

 Data centers have many fast-ramping, often smaller, generators which increase the size of unit commitment UC problems. Further, otherwise ``invisible'' intra-hour volume generated by slow large unit ramping must be explicitly modeled to ensure generation volume  transparency in UC. 
 
 Combined, these effects impact the temporal structure of UC problems. First, tighter and more reliable ramping requires a shift toward finer temporal resolution. There is an ongoing discussion about moving from the traditional 1 hour interval to sub-hourly intervals, possibly as short as 5 minutes~\cite{chen2023,KAZEMI2016338}. 

Increasing temporal granularity causes a proportional growth in problem size, that is, the number of variables and constraints~\cite{Kim20185276, Safdarian20201834}.  Intertemporal coupling leads to a super linear increase in runtime. Most large-scale UC formulations rely on mixed-integer program (\textbf{MIP}) solvers~\cite{achterberg2013}. Their runtime scales poorly as spatial or temporal resolution increases \cite{morrison2016,  ridha2020, berthold2015, koch2022}. 

Existing heuristic decomposition and acceleration methods' performance typically deteriorates when the objective function is strongly temporally coupled, as it is with multi-period UC with ramping constraints. In such settings, solving a single long-horizon UC problem is considerably more demanding than solving multiple shorter horizon problems \cite{Kim20185276, Safdarian20201834}. Therefore, many studies limit  the number of generators, temporal resolution, or planning horizon.

Another class of commercial MIP solvers, such as Gurobi, linearizes the objectives and constraints of the problem, relaxing the MIPs to mixed-integer linear programs %(\textbf{MILP}s) #removed this because we only use it here
~\cite{gurobi}. This allows using heuristics such as branch-and-bound to decrease runtime but their worst case (and often in practice) scaling is still exponential in the number of variables (in this case, proportional to the number of generators) and constraints. Further, the resulting linearization does not capture voltage limits, reactive power constraints, and non-linearities in generator production cost. Our solver uses an interior point method common for continuous problems, which is not limited to linear functions and only requires twice-differentiable functions. In this paper, only our objective is non-linear, but in future work we will introduce transmission constraints in their non-linear (alternating current optimal power flow) form. 

% \textcolor{red}{This paragraph needs to be revised once we have full Gurobi results. -EB.}In earlier work we showed that our solver outperforms a baseline MIP solver by orders of magnitude on runtime and achieves the same results when the problem is sufficiently large. For a simple one-period 276 generator UC problem, with no mechanical constraints, the MIP solver took 10 hours~\cite{regev2026oneperiod}, making it infeasible for the multi-period large system testing in this paper.
% Here, we extend our approach to solving a multi-period problem. We include temporally binding operational constraints and temporally coupled objective function problem, in which costs and operational decisions are explicitly linked across multiple time intervals. We consider intertemporal runtime constraints (minimum run time and maximum daily starts) and ramping constraints.

The main contributions of this paper are:
\begin{itemize}
    \item Introducing a fast, scalable, sub-quadratic algorithm as an alternative to existing exponential algorithms for solving UC with stochastic load.
    \item Showing that in all our tests, our solver is faster than Gurobi, without sacrificing solution accuracy. On the largest problem, it is twice as fast. Our solutions are always feasible, with no constraint violations, and within tolerance of the lower bound on the optimum.
    \item Showing that it is sufficient to know some distribution properties of the upcoming load (as opposed to the exact load profile) to prepare for it well. This means our model only incurs small extra generation costs compared to one which explicitly incorporates future predictions.
    \item Providing analytical guarantees on the accuracy of our algorithm, along with experimental validation.
\end{itemize}
Some limitations on the scope of this study include:
\begin{itemize}
    \item Our grid model does not include transmission constraints (future work). 
    \item Our solver is tailored for UC class of problems and is not a general solution for all MIPs.
\end{itemize}

\cref{sec:multiperiod} introduces our long-horizon planning UC formulation.
\cref{sec:summary} summarizes our results and future research directions. The appendices show supplemental analytical proofs of accuracy and runtime. Appendix \ref{sec:AppendixA} analytically and experimentally shows that the relative error our solver makes vanishes as the problem size increases.

\section{Horizon-Aware Unit Commitment}
\label{sec:multiperiod}
\subsection{Mathematical Formulation}
We expand our previous work~\cite{regev2026oneperiod} include minimum run time, maximum daily starts, and ramping constraints, along with optimizing over generation costs in future time periods. We use the mathematical model in \cref{eq:UnitCommitment} to solve UC, that is, decide which generators to operate and at what capacity to meet demand and contingencies.

\begin{subequations}\label{eq:UnitCommitment}
\begin{align}
\min_{P_i, P_{\delta,i}, u_i} \; &
\sum_{i\in\mathcal{G}} \left(a_i P_i^2 + b_i P_i + c_i\right)u_iu_{t-1,i} +\sum_{i\in\mathcal{G}_d}K_i(u_{t-1,i}(1-u_i)+(1-u_{t-1,i})u_i)
\nonumber\\
+&\gamma\sum_\delta\sum_{i\in\mathcal{G}}\left(a_iP_{\delta,i}^2+b_iP_{\delta,i}+c_i\right)u_i
\label{eq:ObjUC}
\\[6pt]
\text{s.t.} & \max(P_{\min,i},P_{t-1,i}-r_{d,i}) \le P_i   \le\min(P_{\max,i},P_{t-1,i}+r_{u,i})
\label{eq:continuous_vars_UC}
\\&
\forall i\in\mathcal{G}_d : u_i\in \{0,1\}, \;
\forall i\in\mathcal{G}_m : u_i = 1
\label{eq:boolean_vars_UC}
\\&
\sum_{i \in \mathcal{G}_m} P_i + \sum_{i \in \mathcal{G}_d} \left[u_i P_i + (1-u_i)P_{\min,i}\right]u_{t-1,i} 
\ge D-S_r,
\label{eq:Demand_Const_UC}
\\&
\sum_{i} u_i P_{\max,i} \ge D_{\max,48} + 3\sigma_D + R - S_{\max,r} \doteq
 R_{\uparrow}
\label{eq:Max_Const_UC}
\\&
\sum_{i} u_i P_{\min,i}\le D_{\min,48} - \sigma_D - S_{\min,r}\doteq R_{\downarrow}
\label{eq:Min_Const_UC}
\\& 
\sum_{i} u_i P_{\delta,i} \geq D_\delta, \; \forall \delta,
\label{eq:deltaD_UC}
\end{align}
\end{subequations}

The first term in \cref{eq:ObjUC} is the cost of an operating generator. The second term in \cref{eq:ObjUC}  captures the cost of turning a generator off or on. The third term in \cref{eq:ObjUC} biases the solver towards units that will be more efficient in meeting future demand required by \cref{eq:deltaD_UC} (see \cref{sec:results}).

\cref{eq:continuous_vars_UC} limits the production ability of an on generator based on minimum and maximum stable production and ramping rates (for off generators, this variable has no impact). \cref{eq:boolean_vars_UC} defines the operation variables for the generators. \cref{eq:Demand_Const_UC} is a constraint on satisfying demand with current production levels. \cref{eq:Max_Const_UC,eq:Min_Const_UC} define contingencies on the maximum and minimum power that the set of on and ramping up units can generate\footnote{These contingencies are an example. Our model is adaptable to any contingencies.}. 

The extension of our previous work is generator ramping and mechanical constraints and the last term in \cref{eq:ObjUC} and \cref{eq:deltaD_UC}. These balance a UC solution that is efficient now with one that will efficiently meet future demand.  We note that our model can easily handle inverter based resources (\textbf{IBR}s) in its current form. Like demand, IBR production is stochastic and uncontrollable. So, it can easily be modeled as negative demand. 

\cref{tab:variables} details the decision variables and input parameters for the optimization solver. By our convention, all variables and parameters refer to the current time period being solved, unless otherwise noted. $\mathcal{G}_m$ and $\mathcal{G}_d$ are sets of must-run and discretionary generators. We define their union $\mathcal{G}\doteq\mathcal{G}_m \cup \mathcal{G}_d$. We index generators in $\mathcal{G}$ with $i$. 

\begin{table}[htbp]
\centering
\small % Reduces padding slightly to keep the table compact
\begin{tabularx}{\linewidth}{l X}
\toprule
\textbf{Symbol} & \textbf{Description} \\
\midrule
$P_i $, $i \in \mathcal{G}$ & Power output of must-run or discretionary unit $i$ \\
$\delta$ & Demand at specific point in time. The set of $\delta$s considered can be all demand points in the 48 hour time window, or a representative sample. (See *)\\
$P_{\delta,i}$ & Production level by generator $i$ at demand level $\delta$.  \\
$u_i \in \{0, 1\}$, $i \in \mathcal{G}_d$ & Commitment status of discretionary generator $i$. $u_i=0$ means staying off or starting to ramp down. $u_i=1$ means staying on or starting to ramp up. \\
\midrule
$a_i, b_i, c_i$ & Quadratic cost coefficients usually resulting from fuel costs or strategic bidding behavior. \\
$P_{\max,i}, P_{\min,i}$ & Maximum and minimum stable outputs, determined by a unit's physical characteristics. \\
$D$, $\sigma_D$ & Demand volume and standard deviation \\
$K_i$ & Commitment change penalty for generator $i$ (See **)\\
$r_{u,i}$, $r_{d,i}$ & Generator $i$'s ramp up and down limits, normalized such that $\Delta t = 1$. \\
$\gamma$ & Bias factor to select units that will be more efficient at meeting future demand. In our model $\gamma = 1$. \\
$u_{t-1,i}$ & Indicator that is 1 if and only if generator $i$ was ramping up or on in the previous time period \\
$P_{t-1,i}$ & Unit $i$'s production level at the previous time period \\
$S_r$ & Total supply produced by all generators that are already ramping before the solve (controlled). Follows ramping profiles in \cref{fig:ramp_profile_smooth}. \\
$S_{\max,r}$ & Total capacity of ramping up units (ramping down generators will be off before the contingency) \\
$S_{\min,r}$ & Minimum working capacity of ramping up units \\
$D_{\min,48}$, $D_{\max,48}$ & Minimum and maximum volumes of predicted demand in next 48 hours \\
$D_{\delta}$ & Discrete  demand probability sample in next 48 hours \\
$R = \max_{i}P_{\max,i}$ & Reserve margin, used to handle $N-1$ generator outage contingencies\\
\bottomrule
\end{tabularx}
\caption{\label{tab:variables}
Decision variables (first group) and parameters supplied to the optimization solver (second group). *: Our formulation utilizes the cost curves' monotonicity. If a set of units can supply 2 demands efficiently, it can supply any demand between them efficiently. **: The true startup and shutdown costs are rarely the same. Averaging them prevents market distortion and removes artificial incentives to turn on or off.}
\end{table}

\cref{eq:UnitCommitment} is a MIP, which in general are NP-hard. Branch-and-bound methods are typically used to approximate solutions~\cite{morrison2016}. Our algorithm temporarily relaxes $u_i\in \{0, 1\}$ to $y_i \in [0, 1]$ to solve \cref{eq:UnitCommitment} in a similar fashion to branch-and-bound methods, but we use different heuristics that are specific to UC problems. We first re-order discretionary generators $\mathcal{G}_d$ so that \(y_1\geq\ldots\geq y_{|\mathcal{G}_d|}\). We then find the first $\tilde{m}$ discretionary generators such that
\begin{align}
&\sum_{i\in\mathcal{G}_m} P_{\max,i}
+
\sum_{\substack{1\leqslant i \leqslant \tilde m\\ i\in\mathcal{G}_d}} P_{\max,i}
\;\ge\;
R^\uparrow
\label{eq:findmin}
\end{align}

Some  previously committed units must operate due to runtime and power ramping constraints (the set \(\mathcal{G}_m\)). Therefore, the minimum number of generators to operate is $m\doteq |\mathcal{G}_m|+\tilde m$. We can calculate $m_{\max}$, the maximum number of generators that can operate, using \cref{eq:Min_Const_UC} analogously to the way we used \cref{eq:Max_Const_UC} to derive \cref{eq:findmin}. This gives us \cref{eq:findmax}:
\begin{equation}
\sum_{i\in\mathcal{G}_m} P_{\min,i}
+
\sum_{\substack{1\leqslant i \leqslant \tilde{m}_{\max}\\ i\in\mathcal{G}_d}} P_{\min,i}
\;\ge\;
R^\downarrow,
\label{eq:findmax}
\end{equation}
with $m_{\max}\doteq |\mathcal{G}_m|+\tilde m_{\max}$. Note \cref{eq:deltaD_UC} is satisfied by construction, because $\forall \delta: D_{\min,48}\leq D_\delta \leq D_{\max,48}$.
We then solve economic dispatch with all options of generators that can supply the load. Meaning, $\forall k: \; m \le k \le m_{\max}$, we set $y_{j\le k}=1$ and $y_{j>k}=0$ and solve:

\begin{align}
\label{eq:OptimizationEco}
\min_{P_j} \; &
\sum_{j=1}^{k} \left(a_j P_j^{2} + b_j P_j + c_j \right)
\\
\text{s.t.} \;&
\max(P_{\min,i},P_{t-1,i}-r_{d,i}) \le P_i   \le\min(P_{\max,i},P_{t-1,i}+r_{u,i}), \;
\sum_{j=1}^{k} P_j \ge D. \nonumber
\end{align}

Finally, we take the solution with lowest objective. Note that we account for the on/off cost terms in line 3 of \cref{eq:ObjUC} within the rounding scheme and add them to the objective of their respective economic dispatches (\cref{eq:OptimizationEco}), before deciding the minimal objective. \cref{eq:OptimizationEco} only optimizes over the current period, because \cref{eq:UnitCommitment} is only used to decide which units to turn on or off based on projected demand. It does not commit generators to producing $P_\delta$ power when demand $D_\delta$ arrives, it merely ensures they can do so efficiently. 

\cref{alg:ramp} summarizes RRUC. Line 3 is not strictly necessary. However, when we set the linear program (\textbf{LP}) to a tight tolerance of $10^{-9}$, this step helped with numerical precision in the \(y_i\) values, which reduced iterations of the for-loop on line 9. See a bound on the error in Appendix \ref{sec:AppendixA}.

\begin{algorithm}[htbp]
\caption{Runtime and ramp constrained RRUC}
\label{alg:ramp}
\begin{algorithmic}[1]
\For{\textbf{each period}}
    \State Solve \cref{eq:UnitCommitment} with $u_i\in\{0,1\}\rightarrow y_i\in[0,1]$
    \State Fix all solution values except \(y_i\), solve the induced LP
    \State Rank generators in descending order ( \(y_1\geq \ldots\geq y_n\))
    \State \(\ell\gets\min\{k : \sum_{i=1}^kP_{\max,i}\geq R_{\uparrow}\}\)
    \State \(h\gets \max\{k:\sum_{i=1}^kP_{\min,i}\leq R_{\downarrow}\}\)
    \State \(\ell' \gets \max\{k : y_k=1\}\)
    \State \(h' \gets \max\{k : y_k > 0\}\)
    \For{\(k\gets\max(\ell,\ell')\)\textbf{\ to\ }\(\min(h,h')\)}
        \State Solve \cref{eq:OptimizationEco} with generators $1,\ldots,k$
        \State Add state change costs to objective
        \If{objective is lowest so far}
            \State Update best solution and objective
        \EndIf
    \EndFor
    \State Update generator states
    \State {Find must run units (with minimum on time not yet reached or with $P_{t-1,i}-P_{\min,i}$ too large to shut off)}
    \State Find can't start generators (those that have exceeded  maximum allowed daily starts or are already ramping)
    \State \textbf{save} Best solution and best objective
\EndFor
\end{algorithmic}
\end{algorithm}

\begin{figure}[htbp]
\centering
\includegraphics[width=1.0\linewidth]{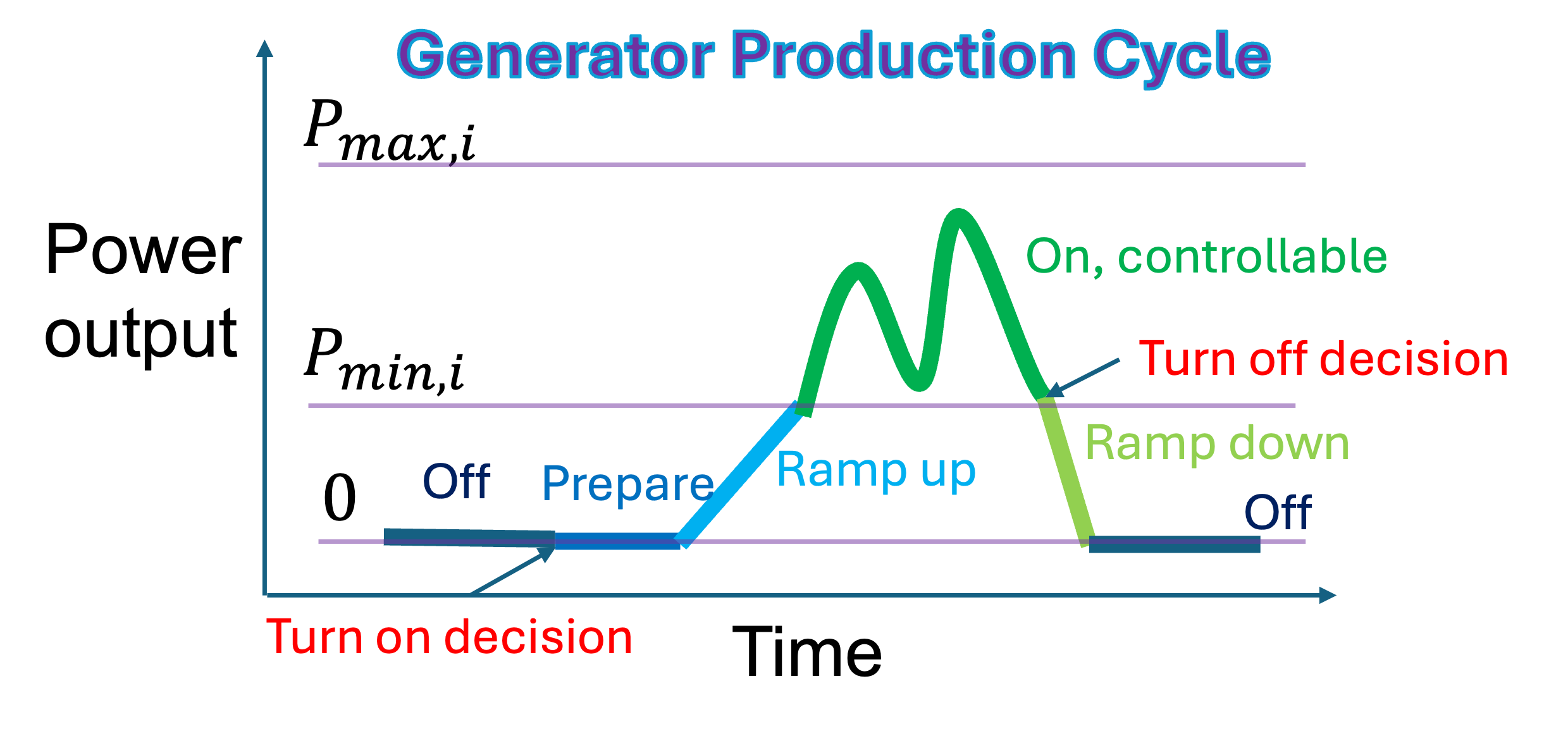}
\caption{\label{fig:ramp_profile_smooth} Unit $i$'s typical production cycle. When on, $P_i\in [P_{\min,i},P_{\max,i}]$, though it cannot exceed its ramping power. It can turn off if $P_{t-1,i}-r_{d,i}\leq P_{\min,i}$. The only other decision can be made when the unit is off, to start ramping up. \cref{alg:ramp} makes these decisions. During other periods, it is uncontrollable.}
\end{figure}

\cref{fig:ramp_profile_smooth} shows a typical production cycle of a unit and the allowed transitions. Each unit can only be in one operating state in any time period, so we round up ramping times from \cite{pjm} to whole periods. Rounding up ensures our solution is a more restrictive UC than the actual UC, meaning it is necessarily feasible\footnote{This is a limitation on using discrete time periods to model a continuous process that is not particular to our solver. Solving UC at smaller time steps reduces this issue, along with others, motivating faster methods such as ours.}.  This means a unit must occupy each state for at least one time period. We use preprocessing to account for the accumulated runtime and ramping periods. We store information about the duration of off, on, or ramping for each unit and use it when solving \cref{eq:UnitCommitment}.
%with respect to ramping constraints

\subsection{Simulation Details}
Our simulations use historical PJM data for load and generating unit bids. PJM recorded its historical load peak on June 23\textsuperscript{rd}, 2025 at 5 pm, with total load of 160.2 GW~\cite{pjm}. We apply this value to scale the system using a 20\% cold reserve margin for resource adequacy. Our demand timeframe is June 18\textsuperscript{th} through June 24\textsuperscript{th}, divided into 5-minute periods. PJM historical data is hourly, so we use linear interpolation to increase the resolution of the data.\par
Our unit selection includes 42 real representative PJM generators and amounts to 9047.9 MW capacity for the entire system. We fitted the piece-wise linear bids from PJM to quadratic cost curves. To scale up the problem, we created slightly modified copies of the original 42 generators. We create these by changing continuous generator parameters such as cost coefficients and minimum and maximum power up to $\pm10\%$ and runtime constraints $\pm 1$ hour randomly (clamping them to at least 1 and at most 24 hours).

To incorporate the runtime constraints, we augmented PJM generator data \cite{pjm} with runtime parameters \cite{diw2013}. PJM data had no clear assignment of units to gas or coal fuel, so we manually classified the units using the following rules: We assigned no units larger than 1 GW gas because there were no operable gas units of this size in PJM territory~\cite{eia8602024}. All units with a minimum runtime of 24 hours are assigned coal, unless their starting price in the supply curve is above \$10/MWh. The cost-based fuel assignment relies on the assumption that coal units have lower operating costs than gas units, although coal prices can be above gas prices on equivalent energy content~\cite{eiam}. 
We adopt cold start values from PJM~\cite{pjm} for our model.

We used approximations based on size and the assumed fuel used by a turbine (gas or coal steam) from two sources~\cite{diw2013,GE}. 
We assigned larger turbines and coal steam turbines  lower ramp values. First, we assigned the ramp percentages to specific generating units. Then, we multiplied $P_{\max}$ of those generators the ramp rate to get the MW/min value we used in the simulation. These values may be lower than physical ramp rates, as a survey of large thermal units suggests that in real systems power plants are often not able to deliver the nominal ramp rates specified for equipment~\cite{India}.

If the system starts with all generators off, \cref{eq:UnitCommitment} is infeasible. We initialized the system by solving a non-ramp-constrained UC instance and then treating the selected generators as if they had been on for the past 24 hours, for the purpose of runtime constraints. If operators are warm starting using \cref{alg:ramp} they can use the current system state.

We implement our method in Julia~\cite{Bezanson2017Julia} with the MadNLP nonlinear programming solver~\cite{Pacaud2024ExaModels}, which built upon~\cite{Regev2022HyKKT, Swirydowicz2025GPU}, with the MA57 linear solver from HSL~\cite{ma57} and the JuMP modeling language~\cite{Dunning2017JuMP}. We use the Gurobi mixed-integer program solver~\cite{gurobi} for comparison. Our computational experiments use an Apple M5 MacBook with 10 cores and 24GB RAM. 

We study a system with 462 units to determine the last term in \cref{eq:ObjUC}'s impact on the overall operation costs. Each 48-hour planning period contains 576 5-minute intervals, each with a load value. To ensure that we can meet this demand efficiently with the units selected for the current period, we include representative demands in the last term of the objective function, \cref{eq:ObjUC}.
Together, the values constitute a load probability distribution over a 48-hour forecast. We use a stochastic collocation method to approximate the probability distribution  with a set of discrete points~\cite{Nobile2008}\footnote{We select demand points using the nodes of the Clenshaw-Curtis quadrature rule~\cite{ClenshawCurtis1960}, mapping each node to its corresponding quantile within the sorted set of discrete 48-hour demand forecasts. We make this choice to include the minimum, maximum, and mean demands, but other choices may work similarly well or better.}.

Due to the long run times required for comparison over hundreds or thousands of time periods at many different scales, we compare our two solvers each to the one commercial solver considered best for that type of problem on a particular dataset, which we scale and perturb randomly to ensure our method is robust. This is the limit of our scope.

\subsection{Results}
\label{sec:results}
\cref{fig:future} shows the cost reduction and runtime tradeoff using a serial implementation of RRUC. The results show an approximately \(27\%\) decrease in cost from the 0 collocation point model (\textbf{CPM}) to 1-CPM, an approximately \(5\%\) lower cost when using 2-CPM and 3-CPM, and negligible advantage from more collocation points. Runtime scales linearly in the number of points used. It is possible that parallelization would lead to further run time reductions.

\begin{figure}[htbp]
\centering
\includegraphics[width=1.0\linewidth]{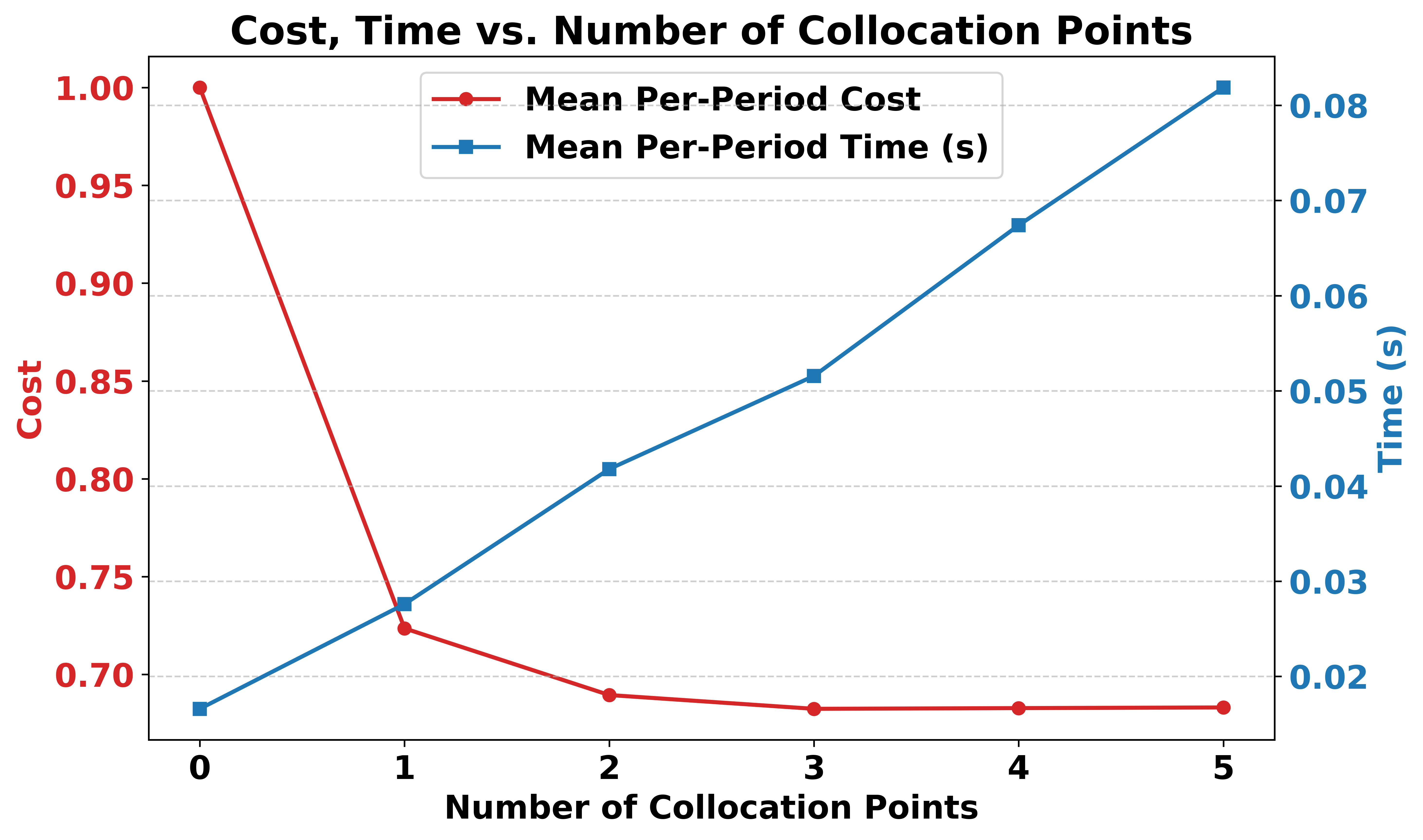}
\caption{\label{fig:future} RRUC's normalized objective and runtime scaling with different CPMs on a 462-generator test case. Runtime increases roughly linearly and the objective levels off. This motivates taking a demand sample $D_\delta$ instead of all demands in \cref{eq:deltaD_UC}.}
\end{figure}

This means that including the mean, minimum, and maximum projected demands in the planning horizon helps the solver reduce overall costs. Including more demands helps a little, but increases solve time a lot. Different applications may select a different tradeoff, but in light of these results, we restrict our attention to 3-CPM of future demands for the rest of the paper.

% \cref{fig:future} shows the cost reduction and runtime tradeoff. Cost decreases $13\%$ just by including the mean as $D_\delta$. Including the minimum and maximum as well decreases costs $17\%$ compared to the baseline. 4-CPM and 5-CPM decreased the objective by less than an additional $1\%$ compared to 3-CPM. Runtime increases roughly linearly in the number of points.

We run \cref{alg:ramp} with 3-CPM for 2016 time periods at 5 minute intervals (7 days) for June 18\textsuperscript{th} to 24\textsuperscript{th}, 2025, which covers the PJM historical peak. \cref{fig:3point} shows that 3-CPM retains accuracy at scale. Its normalized objective decreases slightly with an increase in system size.

\begin{figure}[htbp]
\centering
\includegraphics[width=1.0\linewidth]{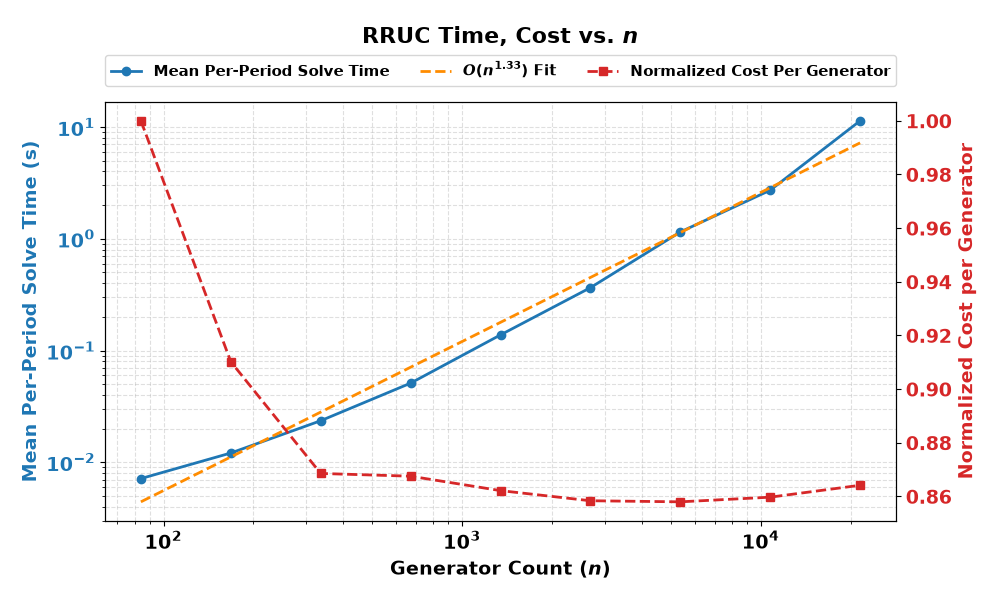}
\caption{\label{fig:3point} Objective and runtime scaling of the 3-CPM from 84 to 21,504 generators. The 3-CPM retains accuracy at scale. RRUC's runtime scales as $\approx O(n^{4/3})$, where \(n\) is the number of units. Cost values per-generator are normalized by the smallest case (84 generators). Demand data is from PJM's historical data from June 18\textsuperscript{th} to 24\textsuperscript{th}, 2025.}
\end{figure}

RRUC's time complexity is dominated by the numerical linear algebra in the continuous optimization we use to solve the relaxed \cref{eq:UnitCommitment}. It is estimated to be $O(n^{1.5})$~\cite{morrison2016}. \cref{alg:ramp} meets and outperforms this scaling slightly. This may be because linear systems generated by \cref{eq:UnitCommitment} are sparser than typical optimization problems due to the $P_{\delta,i}$s which do not interact much with the other variables.\par
To benchmark our speed and accuracy, we compare our results with results obtained by using the Gurobi optimization software~\cite{gurobi} by replacing our solver (lines 1-14 in \cref{alg:ramp}) in our Julia code with a MIP solver from Gurobi's Julia package, using 10 threads for parallelism. As an example, we run with varying generator counts for 7 days each, with a 48-hour planning window. One can choose any planning window without increasing runtime due to CPM sampling. 

\begin{figure}[htbp]
    \centering
    \includegraphics[width=1.0\linewidth]{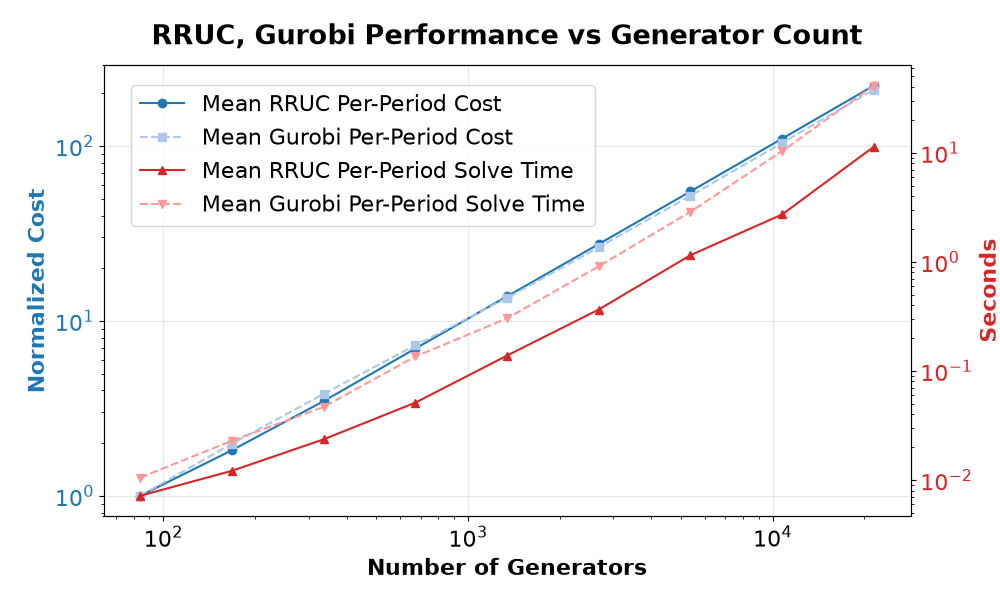}
    \caption{RRUC's solve time and normalized operation cost compared with Gurobi's with  84-21,504 generators. Values are averaged over 2 days (576 5-minute periods). Objective values are normalized relative to the smallest 84-generator case. Both algorithms produce feasible solutions with no constraint violations with objective values identical to the lower bound from the dual problem, but RRUC is faster. This gap increases with more units.}
    \label{fig:gurobi comparison}
\end{figure}

\cref{fig:gurobi comparison} shows that Gurobi and RRUC produce solutions that are just as efficient to our set tolerance of $10^{-4}$. Both solutions are within tolerance of the lower bound of the objective, computed with Gurobi.  RRUC is twice as fast on the largest problem tested using one parallel process compared to Gurobi's 10. Gurobi was slower with fewer parallel processes. Both algorithms scale well in the number of variables, however, Gurobi's runtime scales exponentially in the number of constraints. 

For the largest problem tested, Gurobi took half the time on the 2-CPM that it did on the 3-CPM. Using the 5-CPM, Gurobi took almost four times as long as the 3-CPM. This matches the worst case exponential scaling of branch-and-bound methods, which is not realized when only the number of variables grows and the number of constraints remains constant. This leads us to believe that on AC transmission-constrained problems, where the number of constraints grows linearly in the problem size, RRUC will have a substantial advantage. We save this investigation for future work.

\section{Conclusions and Further Work}
\label{sec:summary}
We extend the RRUC formulation from our previous work \cite{regev2026oneperiod} to \cref{eq:UnitCommitment}, which incorporates mechanical constraints and future cost considerations into UC. We consider representative demands, which we show capture the distribution well. This keeps runtime tractable, even when we increase the time granularity of the problem.

We show that our solver maintains accuracy as the problem size increases and exhibits $\approx O(n^{4/3})$ runtime scaling on our test cases.  On the largest problem we tested, our solver is twice as fast as Gurobi. All solutions are within tolerance of lower bounds on the optimum and Gurobi's solutions and satisfy all constraints.

In future research, we intend to incorporate transmission constraints in UC. We will incorporate batteries, including hydro-storage, in our formulation.
We will attempt to leverage parallelism, including graphical processing units (\textbf{GPU}s), to accelerate our solver. Another possible extension is expanding our approach to include multidimensional models of uncertainty. Finally, we will further examine RRUC's speed-accuracy trade-offs.
\section*{Acknowledgements}
Notice: This manuscript has been authored by UT-Battelle, LLC, under contract DE-AC05-00OR22725 with the US Department of Energy (DOE). The US government retains and the publisher, by accepting the article for publication, acknowledges that the US government retains a nonexclusive, paid-up, irrevocable, worldwide license to publish or reproduce the published form of this manuscript, or allow others to do so, for US government purposes. DOE will provide public access to these results of federally sponsored research in accordance with the DOE Public Access Plan (\href{https://www.energy.gov/doe-public-access-plan}{https://www.energy.gov/doe-public-access-plan}). Research sponsored by the Laboratory Directed Research and Development Program of Oak Ridge National Laboratory, managed by UT-Battelle, LLC, for the US Department of Energy. We thank Nicholson Koukpaizan for his technical review. We thank James Nutaro for helping write the proposal that funded this work.
\begin{appendices}
\section{Asymptotic optimality}
\label{sec:AppendixA}

We show that \cref{alg:ramp} is asymptotically optimal for solving \cref{eq:UnitCommitment} under certain conditions.

\begin{theorem}
    \label{thm: approximation ratio}
    In \cref{eq:UnitCommitment}, set \(\gamma = 0\). Consider a relaxed instance of \cref{eq:UnitCommitment} where \(u_i\in\{0,1\}\leftrightarrow y_i\in[0,1]\). Suppose the continuous stage of \cref{alg:ramp} solves this relaxed instance optimally, and let \(y^*_i\), for \(1\leq i\leq n\), be optimal values of the \(y_i\) variables. Let \(P_i^*\) be optimal values of the \(P_i\) variables.\par
    Suppose there is a feasible solution to \cref{eq:UnitCommitment} such that \(P_i=P_i^*\) and
    $
        \left\{y^*_i = 1\right\}
        \subseteq
        \left\{u_i=1\right\}
        \subseteq
        \left\{y^*_i > 0\right\}
    $.
    That is, \(\{u_i\}\) values are rounded from \(\{y^*_i\}\). Then, let \(\ALG\) be the objective value of the solution to \cref{eq:UnitCommitment} produced by \cref{alg:ramp}. We have
    $
        \ALG \leq {\OPT} + (3 + s)C_{\max}.
    $ with \(\OPT\) the optimal objective value of the (non-relaxed) instance, \(s\) is the number of considered demand levels \(\delta\), and \(C_{\max}\) is the maximum possible coefficient of any \(u_i\) in \cref{eq:ObjUC}. We formally define  
    $
        C_{\max} \doteq \max_{i}\left\{
            \max\left(
                a_iP_{\star,i}^2 + b_iP_{\star,i} + c_i,
                K_i
            \right)
        \right\},
    $ $P_{\star,i}\doteq\min (P_{\max,i}, P_{\min,i}+r_{u,i}+r_{d,i}).$
    % This is because any unit producing more than $P_{\min,i}+r_{d,i}$ is not discretionary, because it cannot ramp down to $P_{\min,i}$ in one step. A unit starting at $P_{i,t-1}\leq P_{\min,i}+r_{d,i}$ can ramp up to $P_i\leq P_{\star,i}$. \(C_{\max}\) is the maximum possible coefficient of any \(u_i\) in \cref{eq:ObjUC}.
\end{theorem}
The definition of \(P_{\star,i}\) comes from the fact that any unit producing more than $P_{\min,i}+r_{d,i}$ is not discretionary, because it cannot ramp down to $P_{\min,i}$ in one step. A unit starting at $P_{i,t-1}\leq P_{\min,i}+r_{d,i}$ can ramp up to $P_i\leq P_{\star,i}$. 
\(C_{\max}\)'s definition is meant to upper bound the coefficient of any \(u_i\) in \cref{eq:ObjUC}, which we will use in the proof of \cref{thm: approximation ratio}.
\begin{proof}
    First observe that, in a relaxed instance, when the non-integer decision variables are fixed, the problem is a linear program with decision variables \(0\leq y_i\leq 1\). We use this fact to bound the maximum number of fractional variables \(0 < y_i^* < 1\) in an optimal solution.\par
    Observe that \cref{eq:UnitCommitment} has \(3 + s\) constraints that are linear in \(u_i\). Applying a standard result in linear programming theory (see, for example, Section 17.2 of \cite{Vazirani2010approx}), there are at least \(n - (3 + s)\) of the variables \(y_i^*\) such that \(y_i^*\in\{0,1\}\). Therefore, there are at most \(3 + s\) of the variables \(y_i^*\) such that \(0 < y_i^* < 1\).\par
    Let \(\OPT_C\) be the minimum objective value of the relaxed instance. Note that \(\OPT_C \leq \OPT\). Note also that the (perhaps infeasible) solution generated by setting \(u_i = 1\) if and only if \(y_i^* = 1\) has an objective value at most \(\OPT_C\), since the objective is non-decreasing in \(u_i\). It follows that any feasible solution generated from rounding the \(3 + s\) fractional \(y_i^*\) values has objective value at most \(\OPT_C + (3 + s)C_{\max}\), since each fractional \(y_i^*\) value that is rounded up can contribute at most \(C_{\max}\) to the final objective value.\par
    Let \(R\) be the feasible solution assumed to exist by the theorem, and let \(u_i^{R}\) be its selection values. Recalling that we assume \(\gamma = 0\), \(R\) will then have objective value
    \begin{equation}\label{eq:ObjBeforeED}
    \begin{aligned}
        &\sum_{i} \left(a_i\left(P_i^*\right)^2 + b_iP_i^* + c_i\right)u_i^Ru_{t-1,i}\\
        +&\sum_{i}K_i\left(u_{t-1,i}\left(1-u_i^R\right)+(1-u_{t-1,i})u_i^R\right)
    \end{aligned}
    \end{equation}
    Recall that \cref{eq:ObjBeforeED} \(\leq\OPT + (3 + s)C_{\max}\). \cref{alg:ramp} will run economic dispatch on rounding \(R\) (since it runs economic dispatch on all roundings), and the resulting solution will be
    \begin{equation}\label{eq:ObjAfterED}
    \begin{aligned}
        &\sum_{i} \left(a_iP_i'^2 + b_iP_i' + c_i\right)u_i^Ru_{t-1,i}\\
        +&\sum_{i}K_i\left(u_{t-1,i}\left(1-u_i^R\right)+(1-u_{t-1,i})u_i^R\right)
    \end{aligned}
    \end{equation}
    where \(P_i'\) are the generation levels set by economic dispatch. Furthermore, since the values \(P_i^*\) are feasible for \cref{eq:UnitCommitment}, they are feasible for \cref{eq:OptimizationEco}, and so therefore the first sum of (\ref{eq:ObjAfterED}) is at most the first sum of (\ref{eq:ObjBeforeED}), since these sums are the objective function minimized by \cref{eq:OptimizationEco}. The second sum is identical between (\ref{eq:ObjBeforeED}) and (\ref{eq:ObjAfterED}), so (\ref{eq:ObjAfterED}) is at most (\ref{eq:ObjBeforeED}).\par
    \cref{alg:ramp} chooses the solution which minimizes the objective value, so \(\ALG\) is at most (\ref{eq:ObjAfterED}). Combining inequalities \\
$\ALG \leq (\ref{eq:ObjAfterED})\leq (\ref{eq:ObjBeforeED})
        \leq \OPT_C + (3 + s)C_{\max}\leq \OPT + (3 + s)C_{\max}
    $
    as desired.
\end{proof}

\begin{corollary}
    \label{thm: relative error}
    Under the conditions of \cref{thm: approximation ratio}, and supposing that \(C_{\max}\) is fixed and \(0 < \OPT \propto n\), 
    $
        (\ALG-\OPT)/\OPT \rightarrow 0,
    $
    as $n\to\infty$.
\end{corollary}
\vspace{0.25cm}
\begin{proof}
    By \cref{thm: approximation ratio},
$
        \frac{\ALG-\OPT}{\OPT} \leq \frac{(3 + s)C_{\max}}{\OPT} \xrightarrow[n\to\infty]{} 0,
    $
    where in the last step we used that \(\OPT \propto n\) and that \(C_{\max}\) is constant with respect to \(n\).
\end{proof}

Recall that we set $\gamma=0$ in the proof. \cref{fig:future} shows that $\gamma>0$ gives better results over multiple time periods (though they may be worse in the current period). This is because we are optimizing which units operate over multiple periods, but the capacity to operate them only over the current period, because we can more easily change it between periods.

In practice, \cref{alg:ramp} produces solutions, on average, with an upper bound on the relative error \(({\ALG-\OPT_C})/{\OPT_C}\) well below the theoretical bound, as shown in \cref{fig:rel_error__vs_n}. The last few points have an average upper bound of 0 up to solver tolerance. They are assigned  a small constant so they are visible on the log scale. 

\begin{figure}[htbp]
    \centering
    \includegraphics[width=1.0\linewidth]{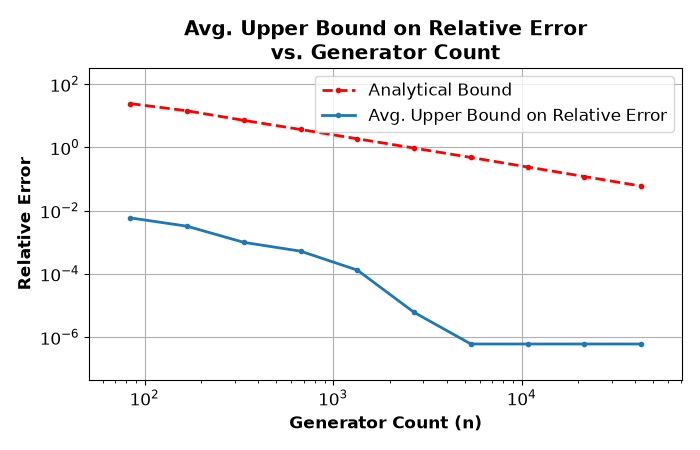}
    \caption{Average upper bound on the relative error from the solution to the relaxed (continuous) \cref{eq:UnitCommitment} and the analytical bound over all periods vs. number of generators. The true error is below both lines.}
    \label{fig:rel_error__vs_n}
\end{figure}

\end{appendices} 

\bibliographystyle{elsarticle-num}
\bibliography{bibfile}
\end{document}